\documentclass{article}

\usepackage{amsmath}
\usepackage{mathtools}
\usepackage{amssymb}
\usepackage{amsthm}
\usepackage{graphicx}
\usepackage{amscd}
\usepackage{epic, eepic}
\usepackage{url}
\usepackage{color}
\usepackage[utf8]{inputenc} 
\usepackage{enumerate}   
\usepackage{graphicx}
\usepackage{epstopdf}
\usepackage{enumerate}

\newtheorem{theorem}{\bf Theorem}[section]
\newtheorem{lemma}[theorem]{\bf Lemma}
\newtheorem{cor}[theorem]{\bf Corollary}

\newtheorem{conj}[theorem]{\bf Conjecture}

\newtheorem{claim}[theorem]{\bf Claim}
\newtheorem{obs}[theorem]{\bf Observation}
\newtheorem{fact}[theorem]{\bf Fact}

\theoremstyle{definition}

\newtheorem{remark}[theorem]{\bf Remark}
\newtheorem{defi}[theorem]{\bf Definition}

\newcommand{\HH}{\mathcal H}

\newcommand{\E}{\mathcal E}

\begin{document}

\title{Coloring linear hypergraphs: the Erd\H{o}s--Faber--Lov\'asz conjecture and the Combinatorial Nullstellensatz}
\author{Oliver Janzer\thanks{Department of Pure Mathematics and Mathematical Statistics, University of Cambridge, United Kingdom.
		E-mail: {\tt oj224@cam.ac.uk}.} \and
	Zoltán Lóránt Nagy\thanks{MTA--ELTE Geometric and Algebraic Combinatorics Research Group,
  E\"otv\"os Lor\'and University, Budapest, Hungary. The author is supported by the Hungarian Research Grant (NKFI) No. K 120154 and SNN 132625 and by the János Bolyai Scholarship of the Hungarian Academy of Sciences. 	E-mail: {\tt nagyzoli@cs.elte.hu}}} 
\date{}

\maketitle

\begin{abstract}

The long-standing Erd\H os--Faber--Lov\'asz conjecture states that every $n$-uniform linear hypergaph with $n$ edges has a proper vertex-coloring using $n$ colors. In this paper we propose an algebraic framework to the problem and formulate a corresponding stronger conjecture. Using the Combinatorial Nullstellensatz, we reduce the Erd\H os--Faber--Lov\'asz conjecture to the existence of non-zero coefficients in  certain polynomials. These coefficients are in turn related to the number of orientations with prescribed in-degree sequences of some auxiliary graphs. We prove the existence of certain orientations, which verifies a necessary condition for our algebraic approach to work.

\end{abstract}

\section{Introduction}

A hypergraph $\HH=(V, \E)$ consists of a nonempty vertex set $V$, and an edge set $\E$. A hypergraph is called {\it linear} if the intersection of each pair of edges contains at most one vertex. 
A {\em  proper vertex coloring} with a color set $C$ of the hypergraph is a function $c: V \rightarrow C$ such that each edge consists of vertices of different colors. 
A  well known conjecture of Erdős, Faber and Lovász, dating back to 1972, asserts an upper bound on the minimum number of colors.

\begin{conj}[Erdős--Faber--Lovász]\label{mainc}
If a linear hypergraph $\HH=(V, \E)$ has $n$ edges, each of size at most $n$, then $\HH$ can be colored properly by $n$ colors.
\end{conj}

Note that the statement follows if one considers only linear hypergraphs such that every vertex is incident to at least $2$ hyperedges, according to the observation below.

\begin{obs} Consider  a hypergraph $\HH=(V, \E)$ in which every edge has size at most $n$. If one deletes the vertices of degree $1$, any proper coloring of the obtained { \em derived  hypergraph} with at least $ n$ colors can be extended to a proper coloring of $\HH=(V, \E)$ with the same color set. 
\end{obs}

Erdős himself considered this one of his three favourite combinatorial problems, and offered one of his highest prizes ever for a proof or disproof \cite{erdos}. By dualizing the problem, its connection to Vizing's theorem becomes clear. To this end, one may assign vertices to the edges  of $\HH$ and introduce the dual hypergraph $\hat{\HH}$ with hyperedges $H_v$ assigned to each vertex $v\in V(\HH)$ such that  $H_v$  consists of the vertices corresponding to hyperedges incident to $v$ in  $\HH$. If  $\HH$ is linear with $n$ edges, then the resulting hypergraph  $\hat{\HH}$  on $n$ vertices is linear as well. Thus this way we get another variant of the conjecture.

\begin{conj}[Erdős--Faber--Lovász, 2nd (dual) variant]\label{mainc2}
Any linear hypergraph on $n$ vertices has chromatic index  at most $n$.
\end{conj}

 The conjecture is confirmed for certain hypergraph families, but  the problem is still widely open, even though asymptotic and fractional versions were established by Kahn and Seymour \cite{Kahn, KS, Seymour}. 
Some notable hypergraph families for which the conjecture is confirmed are the {\em dense} derived hypergraphs for which the minimum degree $\delta(\HH)$ is greater than $\sqrt{n}$ \cite{arroyo}, the uniform derived hypergraphs \cite{Faber} or the cases $n\leq 12$ \cite{Hindman, Romero3} and some other families \cite{Araujo, Faber2, Jackson, Mitchem, Romero2}. These results mostly apply algorithmic and graph theoretic arguments some with computer-based search. For further results on the topic, we refer to \cite{Romero1}.



In this paper we propose an algebraic approach, in connection with the celebrated Combinatorial Nullstellensatz of Alon \cite{Alon}. We point out that the existence of a suitable proper coloring of a hypergraph $\HH$ is strongly connected to the existence of a  particular degree-bounded orientation of certain auxiliary graphs obtained from $\HH$.  In Section $2$ we introduce the algebraic tool and present  two types of auxiliary graphs assigned to the linear hypergraphs. The application of the algebraic tool will imply that if the total sum  of certain signed bounded-degree orientations of the auxiliary graph is nonzero, then there exists a proper coloring with at most $n$ colors. We formulate a conjecture that in fact, this related stronger variant of the Erdős--Faber--Lovász conjecture also holds. 
In Section $3$ we study the strengthened variant of Conjecture \ref{mainc} and confirm it in a weak sense by
showing that a special, so-called Vandermonde-type, orientation exists for  both families of auxiliary graphs assigned to arbitrary $n$-uniform hypergraphs $\HH$  with $n$ edges. This verifies a necessary condition for our algebraic approach to work.  
Finally, in Section $4$ we give some concluding remarks.




\section{The algebraic tool and the strengthening of the E--F--L conjecture}

Our starting point is  Alon's celebrated Combinatorial Nullstellensatz \cite{Alon}, more precisely the Non-vanishing lemma, described below. This tool turned out to be very powerful in several areas of combinatorics; in particular, in graph coloring problems \cite{Alon, Kaul, Murthy, Zhu}.
The connection of graph orientations and this lemma appeared first in the influential paper of Alon and Tarsi \cite{AlonTarsi}.
Here we recall the form of the  Combinatorial Nullstellensatz that we will apply.

\begin{theorem}[Combinatorial Nullstellensatz, Non-vanishing lemma \cite{Alon}]\label{nonvanish}  Let $\mathbb{F}$ be an arbitrary field and let $P=P(x_1, \ldots, x_k)$ be a polynomial of $k$ variables over  $\mathbb{F}$. Suppose that there exists a monomial $\prod_{i=1}^k {x_i^{d_i}}$, such that the sum  $\sum_{i=1}^k {{d_i}}$ equals the total degree of $P$, and the coefficient of $\prod_{i=1}^k {x_i^{d_i}}$ in $P$ is nonzero. Then for any set of subsets $A_1, \ldots, A_k $ of $ \mathbb{F}$ such that $ |A_i| >d_i$, there exists a $k$-tuple $(s_1, s_2, \ldots, s_k)\in \bigtimes A_i$ for which $P(s_1, s_2, \ldots, s_k)\neq 0$.
\end{theorem}

In most applications of the Combinatorial Nullstellensatz, the polynomial can be directly derived from the combinatorial setting, and the choice of the maximal monomial with which Theorem \ref{nonvanish} is applied is also natural.
The main step to make the argument work is to check that the coefficient of this monomial is not zero.  In fact, one usually knows or conjectures in advance the extremal structure, which can be helpful in setting up the corresponding polynomials and verifying that the coefficient in view is nonzero. Unlike in those cases, here we have large freedom to consider a suitable polynomial, and we have to pick the polynomial and its maximal monomial carefully so that the coefficient is surely nonzero. This provides a rather novel application of the main tool.

Let us continue by setting the main notations. For a graph or hypergraph $\HH$, $d(v)$ denotes the degree of the vertex $v$. 
A monomial $\prod_j y_j^{\alpha_j}$ of a multivariate polynomial $Q({\bf y})$ is a {\em $t$-bounded degree monomial} if the degree of each variable $y_j$ is bounded from above by $t$, i.e. $\alpha_j\leq t$. The total degree of a polynomial $Q$ is denoted by $\deg(Q)$.

Our aim is to set up a multivariate polynomial where the variables correspond to vertices of $\HH$ and the values taken by the variables correspond to colors. The polynomial encodes the coloring constraints of the hypergraph $\HH$. In order to do this, we assign an auxiliary graph $G(\HH)$ first to the linear hypergraph $\HH$. We note that we shall propose two different kinds of polynomial that can be used to encode the colouring constraints.

\subsection{Setting up polynomials corresponding to proper colorings}
From now on, speaking about a linear hypergraph $\HH$ we always assume that it has $n$ hyperedges  $\E=\{F_1, F_2, \ldots, F_n\}$ of size $n$, unless specified otherwise. 

We start with introducing {\em two kinds of  auxiliary graph, $G_1(\HH)$  and $G_2(\HH)$ }  assigned to the hypergraph $\HH$. They correspond to two separate approaches to Conjecture \ref{mainc}. Both graphs consist of $n$ vertex-disjoint cliques of size $n$, together with a set of so-called \emph{identifier edges} joining vertices from different cliques. We remark in advance that the identifier edges are not uniquely determined by the hypergraph $\HH$; we have some freedom how to choose them.

\begin{defi}\label{aux1} The   {\em auxiliary graph of first kind } $G_1(\HH)$  assigned to $\HH$ is defined as follows. We take $n$ copies  $K_n^{(i)}$  ($i=1,\ldots, n$) of the complete graph $K_n$, where the  vertices of  $K_n^{(i)}$ are labelled by the vertices  of $F_i$, see Figure \ref{Fig:1}. We call them the {\em  base cliques}. Here each vertex $v\in V(\HH)$ appears $d(v)$ times, and for each $v\in V(\HH)$ we choose an arbitrary spanning tree on the set of those $d(v)$ vertices in $G_1(\HH)$ which are labelled by $v$. We call these \emph{identifier spanning trees}. For each edge of these spanning trees, we define an edge of multiplicity $n-1$ in $G_1(\HH)$, see Figure \ref{Fig:2}. We call these new edges {\em identifier edges}. The edge set of $G_1(\HH)$ consists of the edges of the base cliques and the identifier edges.
\end{defi}

\begin{defi}\label{aux2} The {\em auxiliary graph of second  kind } $G_2(\HH)$  assigned to $\HH$ is defined as follows. We take $n$ copies  $K_n^{(i)}$  ($i=1,\ldots, n$) of the complete graph $K_n$, where the  vertices of  $K_n^{(i)}$ are labelled by the vertices  of $F_i$, see Figure \ref{Fig:1}. We call them the base cliques. Here each vertex $v\in V(\HH)$ appears $d(v)$ times, and for each $v\in V(\HH)$ we choose an arbitrary identifier spanning tree on the set of those $d(v)$ vertices in $G_2(\HH)$ which are labelled by $v$. For each edge $v_{i,j}v_{k,l}$ of the spanning tree, we either take the set $\{v_{i,j}v_{k,t}: t\neq l\}$ or the set $\{ v_{i,t}v_{k,l}: t \neq j\}$ to be edges of $G_2(\HH)$. 
We call these new edges {\em identifier edges } and they have  multiplicity $1$, see Figure \ref{Fig:2}. The edge set of the auxiliary graph consists of the edges of the base cliques and the identifier edges.
\end{defi}


\begin{remark}
    In what follows, when $v_{i,j}v_{k,l}$ is an edge in an identifier spanning tree and $i<k$, then we shall always take the set $\{v_{i,j}v_{k,t}:t\neq l\}$ to be the corresponding edges of $G_2(\HH)$.
\end{remark}

\begin{figure}[h!]
    \centering
    \includegraphics[width=0.6\linewidth]{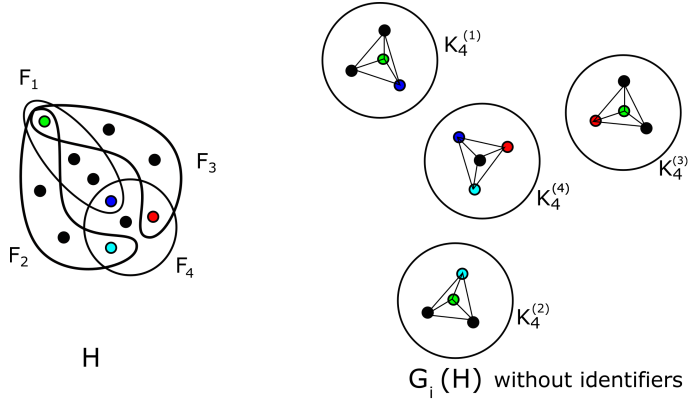}
    \caption{$\HH$ and the base cliques of $G_i(\HH)$, $(i\in \{1,2\})$ }
    \label{Fig:1}
\end{figure}

\begin{figure}[h!]
    \centering
    \includegraphics[width=0.6\linewidth]{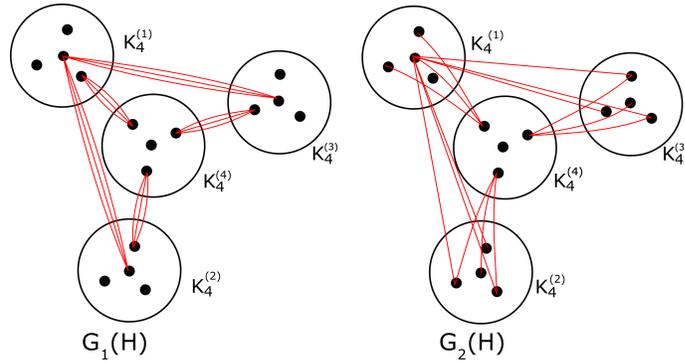}
    \caption{Identifier edges in auxiliary graphs $G_1(\HH)$ and $G_2(\HH)$}
    \label{Fig:2}
\end{figure}

\begin{remark} Informally, the identifier edges are defined as follows. For a pair of intersecting hyperedges in $\HH$, their common vertex $v$ has a copy corresponding to each of the two edges. These two copies are either joined in the corresponding identifier spanning tree or not. If they are, then in $G_1(\HH)$
we put an identifier edge with multiplicity $n-1$ between them, while in $G_2(\HH)$ we put a set of identifier edges forming a star with $n-1$ leaves whose centre is one of the two copies of $v$ and whose leaves are all vertices in the other base clique, apart from the copy of $v$.
\end{remark}

\bigskip

A suitable coloring for $\HH$ with  $n$ distinct elements of  a field $\mathbb{F}$ possesses the following properties:

\begin{itemize}

\item[(P1)] Every clique $F_i$ contains all of the colors (once).

\item[(P2)] Vertices from different cliques which correspond to the same vertex in $\HH$ are assigned the same color.
\end{itemize}

In terms of the auxiliary graphs, a suitable vertex-colouring of $\HH$ corresponds to a vertex-coloring of $G_i(\HH)$ ($i\in \{1,2\}$) in which any two vertices in the same base clique have different colors, while any two vertices joined by an edge in an identifier spanning tree have the same color.

In order to determine whether a given coloring is suitable or not, we assign a variable $x_{i,j}$ to each vertex $v_{i,j}$ ($j=1, \ldots ,n$) of the base cliques $K_n^{(i)}$. These variables will take one of $n$ possible values from $\mathbb{F}$ corresponding to the $n$ colors we can use on the vertices. Moreover, we define the following three families of polynomials.


For any $1\leq i\leq n$, let $$Q_i(\textbf{x})=\prod_{1\leq j<j'\leq n}(x_{i,j}-x_{i,j'}).$$

For any $1\leq i<k\leq n$, if there exist some $v_{i,j}\in K_n^{(i)}$ and $v_{k,l}\in K_n^{(k)}$ which form an edge in an identifier spanning tree, then let

$$R_{i,k}(\textbf{x})=({x_{i,j}-x_{k,l}})^{n-1}-1.$$

(Note that, as $\HH$ is linear, there is at most one such pair of vertices.) Otherwise, let $R_{i,k}(\textbf{x})=1$.

Similarly, if there exist some $v_{i,j}\in K_n^{(i)}$ and $v_{k,l}\in K_n^{(k)}$ which form an edge in an identifier spanning tree, then let

$$\Phi_{i,k}(\textbf{x})=\frac{\prod_{m}(x_{i,j}-x_{k,m})}{x_{i,j}-x_{k,l}},$$ 

and otherwise let $\Phi_{i,k}(\textbf{x})$=1.

Define $P_1(\textbf{x})$ and   $P_2(\textbf{x})$ as follows:

$$ P_1(\textbf{x}) =\prod_{i=1}^n Q_i(\textbf{x})\cdot\prod_{1\leq i<k\leq n}R_{i,k}(\textbf{x}),$$

$$ P_2(\textbf{x}) =\prod_{i=1}^n Q_i(\textbf{x})\cdot\prod_{1\leq i<k\leq n}\Phi_{i,k}(\textbf{x}).$$

Observe that $P_j(\textbf{x})$ is not uniquely determined yet (it depends on the choice of the identifier spanning trees), but its total degree can be expressed by the degree profile of the hypergraph as

  \begin{equation*}\label{eq:deg}
  \deg P_j(\textbf{x})=n\binom{n}{2}+\sum_{v\in V(\HH)}(d(v)-1)(n-1).
   \end{equation*}
   
In order to apply the Non-vanishing lemma (Theorem \ref{nonvanish}), we have to set the field  $\mathbb{F}$. When $n=p$ is a prime, let $P_1(\textbf{x})$ be viewed as a polynomial over $\mathbb{F}=\mathbb{F}_p$. 
For arbitrary $n$, let $P_2(\textbf{x})$  be viewed as a polynomial over $\mathbb{F}=\mathbb{R}$.

Now we can formulate our first contribution which provides an algebraic framework to the main problem.
  
   \begin{conj}\label{polisejtes}
Let $\HH$ be an $n$-uniform linear hypergraph with $n$ edges.
 \begin{enumerate}
 \item[($a$)] One can choose the identifier spanning trees in a way that $P_2(\textbf{x})$   has an $(n-1)$-bounded degree  maximal monomial with nonzero coefficient.
 \item[($b$)] When $n$ is a prime, one can choose the identifier spanning trees in a way that $P_1(\textbf{x})$   has an $(n-1)$-bounded degree  maximal monomial with nonzero coefficient.
 \end{enumerate}
   \end{conj}
   
Observe that this conjecture would imply the Conjecture \ref{mainc} of Erdős, Faber and Lovász.
  Indeed, if  one evaluates the polynomial $P_1$ or $P_2$ on the Cartesian product  $\{0,1,\ldots, n-1\}^n$,
  it will vanish except when the values of the variables correspond to a proper coloring, although one has to suppose that $n$ is a prime in the case of $P_1$. To see this, note that $Q_i(\textbf{x})=0$  holds if  and only if there exist two vertices in some hyperedge $F_i$ with the same color. Moreover, when $n$ is a prime, Fermat's little theorem implies that $R_{i,k}(\textbf{x})=0$ for some $i,k$ if and only if not all identified vertices received the same color. Finally, if $Q_i(\textbf{x})\neq 0$ for every $i$, then we have $\Phi_{i,k}(\textbf{x})=0$ for some $i,k$  if and only if not all identified vertices were  colored with the same color. Hence, $ P_j(\textbf{x})\neq 0$  ($j\in \{1, 2\}$) holds if and only if both properties (P1) and (P2) are satisfied. 
  



\subsection{Searching for a nonzero coefficient and the connection to orientations}

Alon and Tarsi made a connection between a certain coloring problem and the number of Eulerian orientations via the Non-vanishing Lemma \cite{AlonTarsi}. In our case, the situation is somewhat similar.

Define the \emph{sign} of an orientation of the graph $G_a(\HH)$ ($a\in \{1,2\}$) to be $(-1)^t$, where $t$ is the number of edges which point from $x_{i,j}$ to $x_{i,j'}$ with $j<j'$ or from $x_{i,j}$ to $x_{k,l}$ with $i<k$. Now, using the correspondence between $G_a(\HH)$ and $P_a(\textbf{x})$, it is not hard to see that the coefficient of any maximum-degree monomial $\prod_{i,j} x_{i,j}^{\alpha_{i,j}}$ in $P_a(\textbf{x})$ is the sum of the signs of those orientations of $G_a(\HH)$ in which the in-degree of every vertex $v_{i,j}$ is precisely $\alpha_{i,j}$.


To make the Non-vanishing lemma applicable, we clearly need that the exponent of each variable is less than $n$. We will also rely on the following basic fact.

\begin{fact}\label{vander} $(-1)^{\binom{n}{2}}\cdot  Q_i(\textbf{x})=\prod_{j<j'}(x_{i,j'}-x_{i,j})$ equals the Vandermonde determinant 
$$\begin{bmatrix}
1 & x_{i,1} & x_{i,1}^2 & \dots & x_{i,1}^{n-1}\\
1 & x_{i,2} & x_{i,2}^2 & \dots & x_{i,2}^{n-1}\\
1 & x_{i,3} & x_{i,3}^2 & \dots & x_{i,3}^{n-1}\\
\vdots & \vdots & \vdots & \ddots &\vdots \\
1 & x_{i,n} & x_{i,n}^2 & \dots & x_{i,n}^{n-1}
\end{bmatrix}.$$
\end{fact}


Using Fact \ref{vander}, the product $\prod_i Q_i(\textbf{x})\cdot \prod_{i,j} x_{i,j}^{\beta_{i,j}}$ is a linear combination of monomials of the form $\prod_{i,j} x_{i,j}^{\beta_{i,j}+\sigma_i(j)}$, where $\{\sigma_i(j): 1\leq j\leq n\}=\{0,1,\dots,n-1\}$ for every $i$. Thus, to determine the coefficients of the $(n-1)$-degree bounded monomials with maximal degree in $P_1(\textbf{x})$, it suffices to consider those monomials $\prod_{i,j} x_{i,j}^{\beta_{i,j}}$ in $\prod_{i<k} R_{i,k}(\textbf{x})$ which have maximal degree and for which there exist functions $\sigma_i$ as above such that $\beta_{i,j}+\sigma_i(j)\leq n-1$ for every $i,j$. Call such monomials \emph{Vandermonde-completable}. Similarly, to compute the coefficients of the $(n-1)$-bounded degree polynomials in $P_2(\textbf{x})$ with maximal degree, we only need to consider Vandermonde-completable monomials in $\prod_{i<k} \Phi_{i,k}(\textbf{x})$.

We call an orientation of the identifier edge set (in $G_1(\HH)$ or $G_2(\HH)$) {\it Vandermonde-completable} if  one can orient the edges in the base cliques such that each clique spans a transitive tournament and every in-degree in the whole graph is bounded by $n-1$ from above. Under the correspondence between the orientations of the identifier edges in $G_1(\HH)$ and the monomials in $\prod_{i<k} R_{i,k}(\textbf{x})$, Vandermonde-completable orientations correspond to Vandermonde-completable monomials and vice versa. The same holds in the case of $G_2(\HH)$ and $\prod_{i<k} \Phi_{i,k}(\textbf{x})$.

\subsection{Main results}

Our main result states that Vandermonde-completable orientations exist.


\begin{theorem}\label{multiedges}
Let $\HH$ be an 
$n$-uniform linear hypergraph with $n$ hyperedges.  Then, for any choice of the identifier spanning trees, there is a Vandermonde-completable orientation of the identifier edges in $G_1(\HH)$.
\end{theorem}

When $n$ is a prime, the number of Vandermonde-completable orientations corresponding to a given Vandermonde-completable monomial in $\prod_{i<k} R_{i,k}(\textbf{x})$ is not divisible by $n$. Moreover these orientations all have the same sign, so we obtain the following corollary.

\begin{cor} \label{cor:completablemonomial}
Let $n$ be a prime and let $\HH$ be an $n$-uniform linear hypergraph with $n$ hyperedges. Then, for any choice of the identifier spanning trees, there is a Vandermonde-completable monomial (with non-zero coefficient) in $\prod_{i<k} R_{i,k}(\textbf{x})$.
\end{cor}

We also prove the analogue of Theorem \ref{multiedges} in the case of $G_2(\HH)$, although in this case we make a specific choice for the identifier spanning trees.

\begin{theorem}\label{mainP2} Let $\HH$ be an 
$n$-uniform linear hypergraph with $n$ hyperedges.  Then one can choose the identifier spanning trees in a way that there is a Vandermonde-completable orientation of the identifier edges in $G_2(\HH)$.
\end{theorem}

Unfortunately, we cannot prove an analogue of Corollary \ref{cor:completablemonomial} because we cannot compute the sum of the signs of the orientations that yield the same monomial in $\prod_{i<k} \Phi_{i,n}(\textbf{x})$. We leave it as a conjecture.

\begin{conj}
Let $\HH$ be an $n$-uniform linear hypergraph with $n$ hyperedges. Then one can choose the identifier spanning trees in a way that there is a Vandermonde-completable monomial (with non-zero coefficient) in $\prod_{i<k} \Phi_{i,k}(\textbf{x})$.
\end{conj}

\section{The proofs of the main results}

\subsection{The case of auxiliary graph $G_1(\HH)$}

\begin{proof}[Proof of Theorem \ref{multiedges}]

We will use the following claim.

\begin{claim} \label{claim: orient}
    Let $T$ be a tree with $k$ vertices. Write $T'$ for the multigraph on the same vertex set, obtained by taking $n-1$ copies of each edge of $T$. Suppose that for every $v\in V(T)$ there is an integer $\alpha_v$ such that $\alpha_v\leq n-1$ and $\sum_{v\in V(T)} \alpha_v=(k-1)(n-1)$. Then there exists an orientation of the edges of $T'$ in which every $v$ has in-degree $\alpha_v$.
\end{claim}

\begin{proof}
    We use induction on $k$. The statement is clear for $k=0,1$. Assume that $k>1$. Let $w$ be a leaf of $T$. Let $uw$ be the unique edge of $T$ containing $w$. Direct $\alpha_w$ of the $(n-1)$ edges of $T'$ corresponding to $uw$ towards $w$ and direct the rest towards $u$. (Note that the conditions $\alpha_v\leq n-1$ and $\sum_{v\in V(T)} \alpha_v=(k-1)(n-1)$ ensure that $\alpha_w\geq 0$.) Let $S$ be the tree obtained from $T$ by deleting $w$. For every $v\in V(S)\setminus \{u\}$, let $\beta_v=\alpha_v$, while let $\beta_u=\alpha_u-(n-1-\alpha_w)$. Clearly, $\sum_{v\in V(S)} \beta_v=(k-2)(n-1)$. Hence, by the induction hypothesis, we can orient the edges of $S'$ (which is the multigraph obtained by taking $n-1$ copies of each edge of $S$) in a way that the in-degree of every vertex $v\in V(S)$ is $\beta_v$. Together with the orientation of the edges corresponding to $uw$, we get a suitable orientation of $T'$, completing the induction step.
\end{proof}
    
    For any $1\leq a,b\leq n$ with $a\neq b$, let $s(a,b)=
    \begin{cases}
        a-b \mbox{ if } a>b \\
        n+a-b \mbox{ if } a<b
    \end{cases}$
    
    Let $v\in V(\HH)$. Let the edges of $\HH$ which contain $v$ be $F_{i_1},\dots,F_{i_k}$, where $i_1<\dots <i_k$. Then each base clique $K_n^{(i_j)}$ ($1\leq j\leq k$) contains a vertex labelled by $v$, call it $w_j$. Let $T$ be the identifier spanning tree on the vertex set $\{w_1,\dots,w_k\}$. Define $i_{0}$ to be $i_{k}$. Note that $s(i_{j-1},i_j)\leq n-1$ for every $1\leq j\leq k$ and $\sum_{1\leq j\leq k} s(i_{j-1},i_j)=(k-1)n\geq (k-1)(n-1)$. Thus, by Claim \ref{claim: orient}, we can orient the edges of $G_1(\HH)$ corresponding to the spanning tree $T$ in a way that every vertex $w_j$ gets in-degree at most $s(i_{j-1},i_j)$. Hence, the in-degree at every $w_j$ is at most $\max(s(i_q,i_j): q\neq j)$.
    
    Performing this for every $v\in V(\HH)$, we obtain an orientation of the identifier edges in $G_1(\HH)$. We claim that it is Vandermonde-completable. Indeed, for every $1\leq i\leq n$, using the fact that $\{s(q,i):q\neq i\}=\{1,\dots,n-1\}$ and that $\HH$ is linear, the number of vertices of $K_n^{(i)}$ with in-degree at least $n-j$ is at most $j$ for every $1\leq j\leq n$.
\end{proof}

\subsection{The case of auxiliary graph $G_2(\HH)$}


\begin{proof} [Proof of Theorem \ref{mainP2}] For any $x\in V(\HH)$, let the identifier spanning tree corresponding to $x$ be the path $v_{i_1,j_1}v_{i_2,j_2}\dots v_{i_t,j_t}$, where $i_1<i_2<\dots<i_t$ (so the edges of $\HH$ containing $x$ are $F_{i_1},\dots,F_{i_t}$). We call the auxiliary graphs $G_2(\HH)$ obtained this way {\em path-like}. Note that if  $v_{i,j}$ and   $v_{k,l}$ \ $(i<k)$ are   identified, then  the edges between the $i$th and $k$th base clique are joining each  vertex of the $k$th base clique to $v_{i,j}$ except for $v_{k,l}$.

We claim that for this choice, there is a Vandermonde-completable orientation of the identifier edges in $G_2(\HH)$. The proof is by induction on $n$. The case $n=1$ is trivial. For the inductive step, suppose that we have proved the claim for $n$ and let us consider a path-like auxiliary graph $G=G_2(\mathcal{H})$ of a hypergraph $\mathcal{H}$ with $n+1$ edges $F_0, F_1, \ldots ,F_n$, each of size  $n+1$. We call a vertex  $v_{i,j}$ a {\it source} if it forms an edge with  $v_{k,l}$ for some $i<k$ in an identifier spanning path.

\begin{figure}[h!]
    \centering
    \includegraphics[width=0.5\linewidth]{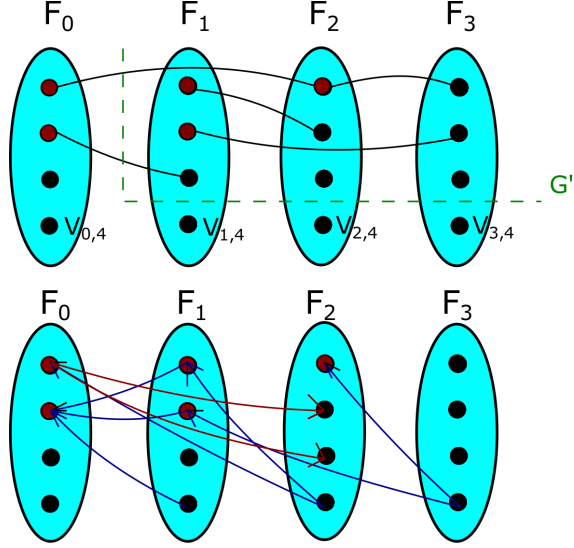}
    \caption{Spanning paths on identified vertices in $G_2(\HH)$ (top) and the orientation of the identifier edges adjacent to $F_0$ or $S$.}
    \label{Fig:3}
\end{figure}

Note that each edge $F_i$ contains a vertex which is not contained in any other edge $F_j$. That is, each base clique contains at least one vertex with no identification to other vertices. Without loss of generality, $v_{1,n+1}\in F_1,\dots,v_{n,n+1}\in F_n$ are such vertices. Let $S=\{v_{1,n+1}, \ldots v_{n,n+1}\}$ and let $\HH'$ be the $n$-uniform linear hypergraph whose edges are $F_1'=F_1\setminus \{v_{1,n+1}\},\dots,F_n'=F_n\setminus \{v_{n,n+1}\}$. Note that the identifier edges of $G_2(\HH)$ induced by the set $F_1'\cup \dots \cup F_n'$ are the identifier edges of $G'=G_2(\HH')$. Hence, by the induction hypothesis, we may orient them in a Vandermonde-completable way. We extend this to a Vandermonde-completable orientation of the identifier edges of $G_2(\HH)$ as follows.

\begin{itemize}
    \item Orient each identifier edge with one endpoint in $S$ towards the other endpoint of the edge (which is necessarily the source). Clearly, the path-like property of $G_2(\mathcal{H})$ implies that the in-degree of any source vertex is increased by at most one, while the in-degree of any non-source vertex is unchanged.
    \item Consider all the identifier edges adjacent to the vertices of $F_0$. Each pair of base cliques $F_0$ and $F_k$ spans either zero or $n$ identifier edges, all of them incident to a single vertex $v_{0,i}$ of $F_0$. Orient these edges $v_{0,i}v_{k,j}$ towards  $v_{0,i}$ if and only if $v_{k,j}$ is a source or $j=n+1$. This does not change the in-degree of the source vertices in $F_1'\cup \dots \cup F_n'$, and it increases the in-degree of any non-source vertex in $F_1'\cup \dots \cup F_n'$ by at most one.
\end{itemize}

It is straightforward to see that the resulting orientation is Vandermonde-completable in the base cliques $F_i$, $i>0$, since each in-degree is increased by at most one and we added a new vertex of in-degree zero to each $F_i$. 
Thus, we only have to confirm that $F_0$ is also  Vandermonde-completable. Since $F_i$ has at most $n-i$ source vertices, any vertex in $F_0$ which is joined to an element of $F_i$ in an identifier spanning tree has in-degree at most $n-i+1$. Moreover, any non-source vertex in $F_0$ has in-degree $0$. Hence, for each $i\geq 1$, the number of vertices in $F_0$ with in-degree at least $n-i+1$ is at most $i$, as required.
\end{proof}

\section{Concluding remarks}

Theorems \ref{multiedges} and \ref{mainP2} suggest that Conjecture \ref{polisejtes} (and, consequently, the Erd\H os--Faber--Lov\'asz conjecture) is likely to hold. 


For some  linear hypergraph families, we can find $(n-1)$-bounded degree monomials which correspond to an (almost) unique orientation of the auxiliary graph, hence their coefficient is non-zero. For example, this is the case when the degree of every vertex in $\HH$ is either $1$ or at least $\sqrt{n}$, or when we have a decent proportion of pairs of hyperedges that do not intersect each other. However, we were unable to extend this approach to be applicable to all linear hypergraphs.



Since the polynomials $P_1$ and $P_2$ have a rather difficult structure and  involve a huge sum of monomials, approaches which provide a simplified sum or an exact formula for the coefficients seem essential to resolve the problem.
As we have seen, Fact \ref{vander} already provides some simplification. We mention that suitable simplified formulae have been obtained in other settings of the Combinatorial Nullstellensatz, see \cite{KP, KNPV, Lason, Uwe}. In these results, the key ingredient was the following  coefficient formula, also mentioned as Quantitative Nullstellensatz.

\begin{lemma}[Coefficient formula]
Let $\mathbb{F}$ be an arbitrary field and 
$P\in \mathbb{F}[x_1, x_2, \dots, x_n]$ a polynomial of degree
$\deg(P)\leq d_1+d_2+\dots+d_n$. 
For arbitrary subsets $C_1, C_2, \dots, C_n$ of $\mathbb{F}$ 
with $|C_i|=d_i+1$, the coefficient of $\prod x_i^{d_i}$ in $P$ is
$$ \sum_{c_1\in C_1} \sum_{c_2\in C_2} \dots \sum_{c_n\in C_n} 
\frac{P(c_1, c_2, \dots, c_n)}{\phi_1'(c_1)\phi_2'(c_2)\dots \phi_n'(c_n)},$$
where $\phi_i(z)= \prod_{c\in C_i}(z-c)$.
\label{interpol}
\end{lemma}


Once we have this coefficient formula, we may seek for a suitable Cartesian product set (or grid) $C_1\times~ C_2 \times ~\dots~ \times~ C_n$ on which polynomial  $P$ vanishes in most cases.
 This approach proved to be successful in several different combinatorial problems, see \cite{KNPV}. To reach an analogous goal, one may take a suitable describing polynomial and suitable monomial (rather than a different grid) in order to guarantee vanishing terms in the sum for the coefficient, via Vandermonde-completability.

\bibliographystyle{amsplain}

\end{document}